\documentclass[reqno]{amsart}

\usepackage{amsmath}
\usepackage{amsfonts}
\usepackage{amsthm}
\usepackage{amssymb}
\usepackage{graphicx}
\usepackage{graphics}
\usepackage{bm}
\usepackage{dsfont}
\usepackage{color}
\usepackage{enumerate}
\usepackage{comment}
\usepackage[font=footnotesize]{caption}
\usepackage[all]{xy}
\usepackage{xy}
\usepackage{tikz}
\usepackage{paralist}
\allowdisplaybreaks
\usepackage{pgfplots}
\pgfplotsset{%
   every tick label/.append style = {font=\tiny},
   every axis label/.append style = {font=\scriptsize}
}

\numberwithin{equation}{section}
\newtheorem{thm}{Theorem}[section]
\newtheorem*{thm*}{Theorem}
\newtheorem*{conj*}{Conjecture}
\newtheorem{cor}[thm]{Corollary}
\newtheorem{lem}[thm]{Lemma}

\theoremstyle{definition}

\newtheorem{dfn}[thm]{Definition}
\newtheorem{rmk}[thm]{Remark}
\newcommand{\N}{\mathds{N}}

\newcommand{\Z}{\mathds{Z}}

\newcommand{\R}{\mathds{R}}

\newcommand{\T}{\mathds{T}}

\newcommand{\diff}{\mathrm{d}}

\newcommand{\p}{\partial}

\newcommand{\kin}{{\mathrm{kin}}}

\begin{document}

\title[Non-resonant circles]{Non-resonant circles for strong magnetic fields on surfaces}

\author[L. Asselle]{Luca Asselle}
\address{Justus Liebig Universit\"at Giessen, Mathematisches Institut \newline \indent Arndtstrasse 2, 35392 Giessen, Germany}
\email{luca.asselle@math.uni-giessen.de}

\author[G. Benedetti]{Gabriele Benedetti}
\address{Universit\"at Heidelberg, Mathematisches Institut, \newline\indent Im Neuenheimer Feld 205, 69120 Heidelberg, Germany}
\email{gbenedetti@mathi.uni-heidelberg.de}

%\date{November, 2019}

\begin{abstract} We study non-resonant circles for strong magnetic fields on a closed, connected, oriented surface and show how these can be used to prove the existence of trapping regions and of periodic magnetic geodesics with prescribed low speed. As a corollary, there exist infinitely many periodic magnetic geodesics for every low speed in the following cases: i) the surface is not the two-sphere, ii) the magnetic field vanishes somewhere. 
\end{abstract}

\maketitle

%%%%%%%%%%%
\section{Introduction}
\subsection{The setting}
Let $M$ be a closed, connected, oriented surface. A magnetic system on $M$ is a pair $(g,b)$, where $g$ is a Riemannian metric on $M$ and $b:M\to\R$ is a function, which we refer to as the magnetic field. A magnetic system gives rise to a second-order differential equation
\begin{equation}\label{e:ODE}
\nabla_{\dot\gamma}\dot\gamma=b(\gamma)\dot\gamma^\perp,
\end{equation}
for curves $\gamma:\R\to M$. Here $\nabla$ is the Levi-Civita connection of $g$ and $\perp$ denotes the rotation of a tangent vector by ninety-degree in the positive direction. Equation \eqref{e:ODE} is Newton's second law for a charged particle on the surface $M$ under the effect of the Lorentz force induced by the magnetic field $b$ \cite{Arn61,AS67}. 

Solutions of \eqref{e:ODE}, also known as magnetic geodesics, have constant speed $s$. Moreover, a curve with constant speed $|\dot\gamma|\equiv s>0$ is a solution of \eqref{e:ODE} if and only if it satisfies the prescribed curvature equation
\begin{equation}\label{e:kappa}
\kappa_\gamma=\frac{b(\gamma)}{s},
\end{equation}
where $\kappa_\gamma:\R\to\R$ denotes the geodesic curvature of $\gamma$ \cite{Arn61,AS67}. From \eqref{e:kappa} we see that the geometric properties of $\gamma$ will depend on the size of the quantity $b/s$. For instance, if we are in the regime of weak magnetic fields, namely if $s$ is large with respect to $b$, the solutions of \eqref{e:ODE} approximate the geodesics of $g$.

In this paper, we are interested in the opposite regime of strong magnetic fields, namely when $s$ is small with respect to $b$. In particular, we will look at the existence of periodic orbits and of trapping regions for the motion when the speed $s$ lies in this regime. Periodic orbits and trapping regions play a crucial role for their applications to the dynamics of charged particles and have been the subject of a vast literature in the last decades. Periodic orbits are usually detected by variational methods (see for instance \cite{Taimanov:1992fs,Ginzburg:1994,Benedetti:2016b} and the discussion below for further references) and are used to study the presence of stable and chaotic trajectories (see \cite{Abreu:2017,Miranda:2007}). Trapping regions yield confinement of particles, a phenomenon which is observed in van Allen belts \cite{Braun}, exploited to build plasma fusion devices \cite{HM} and is mainly studied for magnetic systems on flat three-dimensional space \cite{Truc,Peralta}.

In order to state our main results we need two more pieces of notation and a definition. We denote by $\mu$ the area form of $g$ with respect to the given orientation and by $K:M\to\R$ the Gaussian curvature of $g$.
\begin{dfn}\label{d:non}
A magnetic system $(g,b)$ with $b\not\equiv 0$ is called \textit{resonant at zero speed} if either
\begin{enumerate}
\item both $b$ and $K$ are constant, or
\item $M=S^2$ and one of the following two conditions holds:
\begin{enumerate}[(i)]
\item $b$ is a non-constant, nowhere vanishing function such that the Hamiltonian flow of the function $b^{-2}$ on the symplectic manifold $(S^2,\mu)$ is fully periodic, or
\item $b$ is a non-zero constant, $K$ is non-constant and the Hamiltonian flow of the function $K$ on the symplectic manifold $(S^2,\mu)$ is fully periodic.
\end{enumerate}
\end{enumerate}
\end{dfn}

\begin{rmk}
Fully periodic Hamiltonian flows on $(S^2,\mu)$ are completely classified, up to a change of coordinates. They are all isomorphic to the Hamiltonian flow induced by a constant multiple of the height function of the euclidean unit sphere in $\R^3$ endowed with a constant multiple of the euclidean area form. In particular, the Hamiltonian function possesses exactly two critical points which are additionally non-degenerate. A criterion that prevents a Hamiltonian flow on $(S^2,\mu)$ to be fully periodic is given in \cite[Theorem 1.7]{AB21}.  \qed
\end{rmk}

\begin{rmk}
Given a metric on $S^2$, one can always construct a function $b$ satisfying (i) in Definition~\ref{d:non}. To our best knowledge, it is not known whether examples of metrics whose curvature $K$ satisfies (ii) exist or not. In Theorem~\ref{thm:curvaturenonresonant} we show that there is no such metric of revolution. In \eqref{e:Kres0} we give a geometric characterization of such metrics, and in \eqref{e:Kres} a simple necessary condition that their curvature must satisfy. \qed
\end{rmk}

%%%%%%%%%%%%%%%%%%%%%

\subsection{The results}\label{ss:main}
Our main result asserts that magnetic systems which are not resonant at zero speed possess an abundance of periodic orbits and trapping regions when the speed $s$ is small. Here is the precise statement.
\begin{thm}\label{t:main}
Let $(g,b)$ be a magnetic system on a closed, connected, oriented surface $M$ with $b\not\equiv0$. Suppose that $(g,b)$ is not resonant at zero speed. Then there exists an embedded circle $L\subset \{b\neq 0\}$ such that for all neighborhoods $U$ of $L$ there exist $s_*>0$ and a neighborhood $U'$ of $L$ contained in $U$ with the property that for all $s<s_*$ there holds:
\begin{enumerate}
\item there exist infinitely many periodic magnetic geodesics with speed $s$ contained in $U$ and homotopic to some multiple of the free-homotopy class of $L$,
\item every magnetic geodesic with speed $s$ starting in $U'$ stays in $U$ for all positive and negative times.
\end{enumerate}
\end{thm}

The theorem is proved using a normal form for the flow $\Phi_s$ on $S_sM:=\{(q,v)\in TM\ |\ |v|=s\}$ arising from the tangent lifts of magnetic geodesics with small speed $s$. For the vector field generating $\Phi_s$, the normal form was established by Arnold in \cite{Arn97} and used to show that the magnetic field $b$, or the Gaussian curvature $K$ if $b$ is constant, is an adiabatic invariant for strong magnetic fields \cite[Example 6.29]{AKN}. For the applications to trapping regions and periodic orbits, we need an upgrade of the normal form from the vector field to the Hamiltonian structure inducing the flow $\Phi_s$. This Hamiltonian normal form was established in \cite{Castilho:2001} when $b$ is not constant and in \cite{AB21} when $b$ is constant. We refer to Section~\ref{s:app} for more details and to \cite{BS94} for a related normal form in flat three-dimensional space.
 
The circle $L$ in the statement of Theorem \ref{t:main} is any connected component of a regular level set of the magnetic field $b$, or of the Gaussian curvature $K$ if $b$ is constant, at which the first order of the Hamiltonian normal form is non-resonant (see Section~\ref{ss:res} for a precise definition). The Moser twist theorem \cite{Mos62}, a manifestation in this setting of the celebrated KAM theorem, implies the existence of nearby $\Phi_s$-invariant two-tori inside $S_sM$ filled with quasi-periodic 
motions. Considering two of these tori lying to the left and to the right of the circle $L$, one can trap trajectories starting inside a small neighborhood of $L$. Moreover, $\Phi_s$ admits a closed annulus-like surface of section in the region between the two tori. The corresponding first-return map is twist by the non-resonance condition and the existence of infinitely many periodic orbits follows by the Poincar\'e--Birkhoff theorem \cite{Gole:2001vx}.

If $K$ and $b\not\equiv0$ are constant, then all magnetic geodesics are periodic for $s$ small and every embedded circle in $M$ has the properties stated in Theorem~\ref{t:main}. Thus the only instance where the existence of an embedded circle with those properties is not known is case (2) in Definition~\ref{d:non}. Moreover, in many situations we can give additional information on the position of such an embedded circle.
\begin{thm}\label{t:main2}
Let $(g,b)$ be a magnetic system on $M$. Suppose that one of the following conditions is satisfied:
\begin{enumerate}[(i)]
\item $b\not\equiv 0$ and the set $\{b=0\}$ is non-empty;
\item the function $b$ has a degenerate strict local maximum or minimum $z_0\in\{b\neq0\}$;
\item the function $b$ attains a strict local maximum or minimum at an embedded circle $L'\subset\{b\neq0\}$.
\end{enumerate}
Then there exists a fundamental system $\{U_i\}_{i\in \N}$ of neighborhoods of $\{b=0\}$, respectively of $z_0$ or $L'$, such that $\p U_i$ is the union of embedded circles satisfying the properties of Theorem~\ref{t:main} for all $i\in\N$. In particular, each of the sets $\{b=0\}$, $z_0$ and $L'$ is stable in the sense of Lyapunov.

If in (ii) and (iii) we additionally assume that $z_0$ or $L'$ are isolated in the set of critical points of $b$, then $U_i$ can be taken to be a disc around $z_0$ or a tubular neighborhood of $L'$.

Same conclusions as above hold if $K$ satisfies the analogous of conditions (ii) and (iii) and $b$ is a positive constant in a neighborhood of $z_0$ and $L'$.
\end{thm}
\begin{rmk}
Case (ii) in Theorem~\ref{t:main2} under the hypotheses that $z_0$ is isolated in the set of critical points of $b$ was treated in \cite[Corollary 1.5]{AB21}. Case (iii) under the hypothesis that $L'$ is isolated in the set of critical points of $b$ was treated in \cite[Corollary 1.3]{Castilho:2001}. On the other hand, notice that \cite[Corollary 1.2]{Castilho:2001} about a non-degenerate local maximum or minimum point of $b$ is not correct since such a point might be resonant.\qed
\end{rmk}

\begin{rmk}\label{r:saddle}
The presence of a non-degenerate saddle point $z_1\in M$ for $b$ with $b(z_1)\neq0$ also forces the existence of embedded circles as in Theorem~\ref{t:main} intersecting the four components of $\{b\neq b(z_1)\}$ in an arbitrarily small neighborhood of $z_1$. A partial result in this direction was obtained in \cite[Theorem 1.6]{AB21}. As a consequence, trajectories starting close to $z_1$ can only escape along the four branches of $\{b=b(z_1)\}$. We refer to Theorem~\ref{t:saddlemag} for the precise formulation of the result. A similar statement holds if $b$ is a non-zero constant and $z_1$ is a non-degenerate saddle point of $K$. \qed
\end{rmk}

We end this introduction discussing the problem of existence of periodic magnetic geodesics for strong magnetic fields in more detail. We start by observing that Theorem~\ref{t:main} has the following consequence.

\begin{cor}\label{c:nons2}
Let $(g,b)$ be a magnetic system with $b\not\equiv 0$. If $M\neq S^2$ or $b$ vanishes somewhere, then there is $s_*>0$ such that there are infinitely many 
periodic magnetic geodesics with speed $s$ for every $s\in (0,s_*)$.\qed
\end{cor}

Let us examine how Corollary~\ref{c:nons2} complements previous results in the literature. When $b\not\equiv 0$, the existence of one periodic magnetic geodesic for all low speeds follows from \cite{Ginzburg:1987lq} (see also \cite{AB21}). If the function $b$ assumes both positive and negative values, which happens for instance if the two-form $b\mu$ is exact, then the existence of infinitely many periodic magnetic geodesics for \textit{almost every} speed $s<\tilde s$ was proved in the series of papers \cite{Abbondandolo:2017,Abbondandolo:2015lt,Abbondandolo:2014rb,Asselle:2015ij,Asselle:2015hc} by employing the variational characterization of periodic solutions to \eqref{e:ODE} as critical points of the free-period action functional. The value $\tilde s$ can be explicitly characterized and, when $b\mu$ is exact, $\tfrac12{\tilde s}^2$ coincides with the Ma\~n\'e critical value of the universal cover.
These methods have failed so far to yield existence of infinitely many periodic orbits for every small speed because of the lack of good compactness properties of the functional (namely, the lack of the Palais--Smale condition) and the existence for almost every speed is recovered by means of Struwe's monotonicity argument \cite{Struwe:1990sd}.

Our results upgrade the almost everywhere existence results to \textit{every} low speed, even though a priori only on a smaller speed range since we do not have an explicit estimate on the value $s_*$ appearing in Corollary~\ref{c:nons2}. Thus the existence of infinitely many periodic orbits for $s<\tilde s$ remains an open problem (the existence of one periodic orbit being known by \cite{Taimanov:1992fs,Contreras:2004lv}). We shall notice that the almost everywhere existence results extend to the more general class of Tonelli Lagrangians, i.e.~fibrewise strictly convex and superlinear Lagrangians, see \cite{Asselle:2016qv}. In contrast, the arguments of the present paper seem to carry over to the Tonelli setting only if 
the energy function of the Lagrangian attains a Morse--Bott minimum at the zero section of $TM$.

If $b$ is nowhere vanishing, the existence of infinitely many contractible periodic magnetic geodesics for all low speeds was proven for $M=\T^2$ in \cite{Ginzburg:2009} and for $M$ with genus at least $2$ in \cite{Ginzburg:2014ws} in the context of the Conley conjecture \cite{Ginzburg:2015}. In contrast, the periodic orbits found in Corollary~\ref{c:nons2} are all homotopic to the iterate of some free-homotopy class but in general will not be contractible. If $M=S^2$ and $b$ is nowhere vanishing, for every low speed there are either two or infinitely many periodic magnetic geodesics \cite{Benedetti:2014}. Examples with exactly two periodic orbits at a single speed were constructed in \cite{Benedetti:2016}. However, it is still an open problem to determine if magnetic systems on $S^2$ which are resonant at zero speed have infinitely many periodic orbits at every low speed or at least for a sequence of speeds converging to zero. The latter option holds for rotationally symmetric magnetic systems as we observe in Corollary~\ref{c:rotinf}.

\subsection*{Organization of the paper} The next sections are organized as follows:
\begin{itemize}
\item In Section~\ref{s:nonresonant} we discuss the existence problem of non-resonant circles for autonomous Hamiltonian systems on a surface with finite area. The main technical result of this section is Lemma~\ref{l:main}.
\item In Section~\ref{s:Moser} we prove Theorem~\ref{t:Moser} which combines the Moser twist theorem and the Poincar\'e--Birkhoff theorem to produce trapping regions and infinitely many periodic orbits close to a non-resonant circle for small non-autonomous perturbations of autonomous Hamiltonian systems on surfaces.
\item In Section~\ref{s:app}, we apply the abstract results of the previous section to magnetic systems using the normal form of \cite{Castilho:2001,AB21} and prove Theorems~\ref{t:main} and~\ref{t:main2}.
\item In Section~\ref{s:resonant} we discuss some results about magnetic systems on $S^2$ which are resonant at zero speed.
\end{itemize}

\subsection*{Acknowledgments} L.A.~is partially supported by the Deutsche Forschungsgemeinschaft under the DFG-grant 380257369 (Morse theoretical methods in Hamiltonian dynamics). G.B.~is partially supported by the Deutsche Forschungsgemeinschaft 
under Germany's Excellence Strategy
EXC2181/1 - 390900948 (the Heidelberg STRUCTURES Excellence Cluster), under the Collaborative Research Center SFB/TRR 191 - 281071066 (Symplectic Structures in Geometry, Algebra and Dynamics) and under the Research Training Group RTG 2229 - 281869850 (Asymptotic Invariants and Limits of Groups and Spaces). We are grateful to Viktor Ginzburg and Leonardo Macarini for their comments on the preliminary version of the manuscript.

%%%%%%%%%%%%%%%%%%%
%%%%%%%%%%%%%%%%%%%
%%%%%%%%%%%%%%%%%%%

\section{Non-resonant circles for autonomous systems on symplectic surfaces}
\label{s:nonresonant}

In this section $N$ will denote a connected surface without boundary, not necessarily compact. Moreover, $\omega$ will denote a symplectic form on $N$ and $H:N\to\R$ a proper Hamiltonian function. A set $U\subset N$ is said to be a neighborhood of $\{H=+\infty\}$, if $H$ is bounded from above on $N\setminus U$. A similar property defines neighborhoods of $\{H=-\infty\}$. We denote by $X_H$ the associated Hamiltonian vector field determined by the equation $\omega(X_H,\cdot)=-\diff H$.
\subsection{Orbit cylinders} In this first subsection, we introduce orbit cylinders which will be the main protagonist of the crucial Lemma~\ref{l:main} in the next subsection.
\begin{dfn}
Let $\gamma:\R\to N$ be a non-constant periodic orbit for $H$. We call the support $L$ of $\gamma$, namely $L=\gamma(\R)$, an \textit{embedded circle} of $H$ at level $c:=H(\gamma)$.
\end{dfn}
\begin{rmk}
Since $H$ is a first integral of motion and it is proper, its embedded circles are exactly the regular connected components of the level sets of $H$. \qed
\end{rmk}
\begin{dfn}
An \textit{orbit cylinder} (of $H$) is a smooth isotopy $c\mapsto L_c\subset N$ parametrized over an interval $I\subset\R$ such that $L_c$ is an embedded circle of $H$ at level $c$ for every $c\in I$. We say that an orbit cylinder $c\mapsto L_c$ passes through an embedded circle $L$ of $H$ if $L=L_{c_0}$ for some $c_0$ in the interior of $I$.
\end{dfn}
\begin{rmk}\label{r:imp}
By the implicit function theorem, for every embedded circle $L$ of $H$ at level $c_0$ there exists a unique orbit cylinder with $I=(c_0-\epsilon,c_0+\epsilon)$ for small $\epsilon>0$ passing through it. \qed
\end{rmk}
\begin{dfn}
We say that an orbit cylinder $c\mapsto L_c$ is bigger than another orbit cylinder $c\mapsto L'_c$, if the former isotopy extends the latter. We call an orbit cylinder \textit{maximal} if it is maximal with respect to such partial order. 
If $c\mapsto L_c$ is a maximal cylinder passing through $L=L_{c_0}$, we call the restriction of $c\mapsto L_c$ to $I\cap[c_0,+\infty)$ the \textit{maximal forward orbit cylinder} starting at $L=L_{c_0}$. Analogously, we call the restriction of $c\mapsto L_c$ to $I\cap(-\infty,c_0]$ the \textit{maximal backward orbit cylinder} ending at $L=L_{c_0}$.
\end{dfn}
\begin{lem}\label{l:emb}
Let $L$ be an embedded circle of $H$. Then there exists a unique maximal orbit cylinder passing through $L=L_{c_0}$ and it is parametrized over an open interval $I=(c_-,c_+)$. If $c_+<+\infty$, then the closure of $\cup_{c\in[c_0,c_+)}L_c$ contains a critical point $z_\infty$ of $H$. A similar statement holds when $c_->-\infty$.
\end{lem}
\begin{proof}
Let $V\subset N$ be the open set of regular points of $H$. On $V$ we define a flow $t\mapsto\Psi^t$ such that 
\begin{equation}\label{e:dec}
	H\circ\Psi^t(z)=H(z)+t,\qquad \forall\,t\in(t^-(z),t^+(z)),
\end{equation}
where $(t^-(z),t^+(z))$ is the maximal interval of definition of the flow at $z\in V$. For instance $\Psi^t$ can be obtained integrating the vector field $X=\nabla H/|\nabla H|^2$ where the norm and the gradient are taken with respect to any Riemannian metric on $V$. If $t^+(z)<+\infty$, then the closure of the set $S=\{\Psi^t(z)\ |\ t\in[0,t^+(z))\}$ is contained in the set $H^{-1}([H(z),H(z)+t^+(z)])$ and hence is compact in $N$ since $H$ is proper. As $[0,t^+(z))$ is the maximal interval of definition of the trajectory $t\mapsto\Psi^t(z)$, the closure of $S$ cannot be contained in $V$ and therefore there exists a critical point $w_z\in N$ of $H$ lying in it.

Let now $L$ be an embedded circle for $H$ at level $c_0$. We define
\begin{equation}\label{e:cpm}
c_+:=c_0+\inf_{z\in L}t^+(z),\qquad c_-:=c_0+\sup_{z\in L}t^-(z)
\end{equation}
and the orbit cylinder $c\mapsto L_c:=\Psi^{c-c_0}(L)$ for $c\in(c_-,c_+)$ which is well-defined by \eqref{e:dec} and \eqref{e:cpm}. If $c_+<\infty$, then the lower semi-continuity of $t^+:V\to\R$ implies the existence of some $z\in L$ such that $c_+=c_0+t^+(z)$. It follows that there exists a critical point $z_\infty$ of $H$ in the closure of $\cup_{c\in[c_0,c_+)}L_c$. A similar argument works for $c_->-\infty$. This shows that $c\mapsto L_c$ is a maximal orbit cylinder. Uniqueness follows from the local uniqueness of orbit cylinders passing through a given embedded circle as observed in Remark~\ref{r:imp}. 
\end{proof}
\subsection{Resonant cylinders}\label{ss:res}
The goal of this subsection is to define resonant cylinders and prove their properties in Lemma~\ref{l:main}. This result will be used in the next subsection to give criteria for the existence of non-resonant cylinders.
\begin{dfn}
Let $c\mapsto L_c, \ c\in I,$ be an orbit cylinder.  We denote by $T(c)$ the period of the orbit of $H$ supported on $L_c$. We say that a circle $L=L_{c_0}$ with $c_0\in I$ is \textit{non-resonant}, if
\[
\frac{\diff T}{\diff c}(c_0)\neq0
\]
and \textit{resonant} otherwise. If $L_c$ is resonant for all $c\in I$ then we say that the orbit cylinder is resonant. If there is a $c_0\in I$ such that $L_{c_0}$ is non-resonant then we call the whole orbit cylinder non-resonant.
\end{dfn}
We recall now the classical relation connecting the period function $T$ to the area function of sub-cylinders. To this purpose consider two parameters $c_1,c_2\in I$. If $c_1\leq c_2$ the set
\[
U_{c_1,c_2}:=\bigcup_{c\in(c_1,c_2)}L_c
\]
is diffeomorphic to an open cylinder inside $N$ and we endow it with the orientation induced by $\omega$. For $c_2\leq c_1$, we define $U_{c_1,c_2}$ as $U_{c_2,c_1}$ with the opposite orientation. We fix $c_0\in I$ and define the area function
\[
A:I\to\R,\qquad A(c):=\int_{U_{c_0,c}}\omega.
\]
\begin{lem}\label{l:ap}
There holds
\[
\frac{\diff A}{\diff c}(c_1)=T(c_1),\qquad\forall\,c_1\in I.
\]
\end{lem}
\begin{proof}
	Let us consider an orientation-preserving parametrization $F:I\times\T\to N$ of the orbit cylinder, so that $\theta\mapsto F(c,\theta)$ parametrizes $L_c$ and $H(F(c,\theta))=c$. Then there exists a function $a:I\times\T\to(0,\infty)$ such that $\p_\theta F=aX_H$ and $\diff H(\p_c F)=1$. If $t$ is the time parameter of $X_H$ along $L_{c_1}$, then $\diff t(\p_\theta F)=a(c_1,\cdot)$ and
	\[
	T(c_1)=\int_{\T}\diff t(\p_\theta F)\,\diff\theta=\int_{\T}a(c_1,\theta)\,\diff\theta.
	\]
	To compute the area, we observe that
	\[
	\omega(\p_c F,\p_\theta F)=a\omega(\p_c F,X_H)=a\diff H(\p_c F)=a.
	\]
	Therefore, we have
	\[
	A(c_1)=\int_{c_0}^{c_1}\int_{\T}\omega(\p_c F,\p_\theta F)\,\diff c\,\diff\theta=\int_{c_0}^{c_1}\Big(\int_{\T}a(c,\theta)\,\diff\theta\Big)\diff c
	\]
	and by the fundamental theorem of calculus
	\[
	\frac{\diff A}{\diff c}(c_1)=\int_\T a(c_1,\theta)\,\diff\theta=T(c_1).\qedhere
	\]
\end{proof}

We are ready to prove the main lemma.

\begin{lem}\label{l:main}
Let $(N,\omega)$ be a symplectic surface without boundary having finite symplectic area. Let $H:N\to\R$ be a proper Hamiltonian. Let $c\mapsto L_c$ be a forward maximal orbit cylinder with $c\in[c_0,c_+)$ starting at some $L=L_{c_0}$. If this orbit cylinder is resonant, then $c_+<+\infty$ and there exists a non-degenerate local maximum point $z_\infty\in N$ of $H$ such that $D:=U_{c_0,c_+}\cup\{z_\infty\}$ is an embedded disc in $N$. In particular $H$ induces a Hamiltonian circle action on $D$ with $z_\infty$ as the only fixed point. A similar statement holds for resonant backward maximal cylinders.	
\end{lem}
\begin{proof}
Since the orbit cylinder $c\mapsto L_c$ is resonant, there exists a smallest positive $\tau$ such that
\begin{equation}\label{e:period}
\Phi^\tau_H(w)=w,\qquad \forall\,c\in[c_0,c_+),\ \forall\, w\in L_c.
\end{equation}
Since $\tau$ is the common period of all the orbits in the cylinder, Lemma~\ref{l:ap} implies
\[
\tau (c-c_0)=\int_{U_{c_0,c}}\omega\leq \int_N\omega<+\infty
\]
which shows that $c_+<+\infty$. By Lemma~\ref{l:emb}, there exists a sequence $c_n\to c_+$ and a sequence $z_n\in L_{c_n}$ such that $z_n$ converges to a critical point $z_\infty$ of $H$. We claim that
\begin{equation}\label{e:dist}
\sup_{z\in L_{c_n}}\mathrm{dist}(z,z_\infty)\to0,\quad \text{for }n\to\infty.
\end{equation}

Let us assume this were not the case. Since $H$ is proper, the set $\bigcup_{n\in\N}L_{c_n}$ is contained in a compact set of $N$. This fact together with the negation of \eqref{e:dist} implies that we can extract a subsequence, which we still label by $n$, with the following property: There exists $w_\infty\in N\setminus\{z_\infty\}$ and a sequence $w_n\in L_{c_n}$ such that 
\[
\lim_{n\to\infty}w_n=w_\infty.
\]
By \eqref{e:period}, there is a sequence $\tau_n\in[0,\tau)$ such that $w_n=\Phi^{\tau_n}_H(z_n)$. Up to extracting a subsequence, we can assume that $\tau_n\to\tau_\infty\in[0,\tau]$ as $n\to\infty$. By the continuous dependence of the flow $\Phi_H$ on the initial condition and the fact that $z_\infty$ is a rest point of $\Phi_H$, we get the contradiction
\[
z_\infty=\Phi^{\tau_\infty}_H(z_\infty)=\lim_{n\to\infty}\Phi^{\tau_n}_H(z_n)=\lim_{n\to\infty}w_n=w_\infty.
\]

Let us now consider a chart $\varphi:U_\infty\to\R^2$ centered at $z_\infty$. By \eqref{e:dist} the circle $L_{c_n}$ is contained in $U_\infty$ for $n$ large enough and therefore $\varphi(L_{c_n})$ is an embedded circle in $\R^2$. By the Jordan curve theorem $\varphi(L_{c_n})$ bounds a compact region homeomorphic to a disc $D_n$. By \eqref{e:dist}, there holds
\begin{equation}\label{e:dist0}
\sup_{z\in \varphi^{-1}(D_n)}\mathrm{dist}(z,z_\infty)\to0,\quad \text{ for }n\to\infty.
\end{equation}
Indeed, the Euclidean distance in the chart is equivalent to the distance function on $N$ and therefore $D_n$ is contained in a euclidean ball of radius smaller than a fixed constant times $\sup_{z\in L_{c_n}}\mathrm{dist}(z,z_\infty)$. As a consequence,
\begin{equation}\label{e:area0}
\lim_{n\to\infty}\int_{\varphi^{-1}(D_n)}\omega=0.
\end{equation}

The next step is to prove the following claim: There exists $m$ such that for all $n\geq m$, the circle $L_c$ is contained in the interior of $\varphi^{-1}(D_n)$ for all $c\in(c_n,c_+)$. If this were not the case, then $U_{c_0,c_n}\subset\varphi^{-1}(D_n)$ for some $n$ that can be chosen arbitrarily large. Therefore, 
\[
\int_{\varphi^{-1}(D_n)}\omega\geq \int_{U_{c_0,c_n}}\omega\geq \int_{U_{c_0,c_1}}\omega,
\]
which contradicts \eqref{e:area0} for $n$ large enough and proves the claim.

To prove that $D:=\{z_\infty\}\cup U_{c_0,c_+}$ is a neighborhood of $z_\infty$ homeomorphic to a disc is enough to show that $\varphi^{-1}(\mathring D_m)$ is equal to $\{z_\infty\}\cup U_{c_m,c_+}$. By the claim, $\{z_\infty\}\cup U_{c_m,c_+}$ is contained in $\varphi^{-1}(\mathring D_m)$. If $z$ were a point of $\varphi^{-1}(\mathring D_m)$ which is not in $\{z_\infty\}\cup U_{c_m,c_+}$, then $\varphi(L_{c_m})$ is not homotopic to a constant in $D_m\setminus\{\varphi(z)\}$. However, $\varphi(L_{c_m})$ is also homotopic in $D_m\setminus\{\varphi(z)\}$ to $\varphi(L_{c_n})$ for all $n$ large enough and $\varphi(L_{c_n})$
is homotopic to a constant in $D_m\setminus\{\varphi(z)\}$ by \eqref{e:dist}, a contradiction. This shows that $D$ is an open, embedded disc in $N$ and that $\Phi_H$ induces a Hamiltonian circle action on $D$ with unique fixed point $z_\infty$. By \cite[Lemma 5.5.8]{McDuff:1998uu}, $z_\infty$ is a non-degenerate local maximum of $H$.
\end{proof}
\subsection{Non-resonant cylinders}
In this subsection, we exploit Lemma~\ref{l:main} to prove the existence of non-resonant cylinders in certain situations. These results will be used in the proofs of Theorems~\ref{t:main} and~\ref{t:main2}.

\begin{thm}
\label{t:dichotomy}
Let $(N,\omega)$ be a symplectic, connected surface without boundary having finite symplectic area. Let $H:N\to\R$ be a non-constant proper Hamiltonian. Then the set of maximal orbit cylinders is non-empty and
exactly one of the following statements holds:
	\begin{enumerate}[(i)]
		\item All maximal orbit cylinders are non-resonant.
		\item There is only one maximal orbit cylinder. Such a cylinder is resonant and $N$ is diffeomorphic to $S^2$. In this case, $N$ is the union of the cylinder and the unique maximum $z_\infty$ and minimum point $z_{-\infty}$ of $H$. The points $z_\infty$ and $z_{-\infty}$ are non-degenerate and $\Phi_H$ induces a Hamiltonian circle action on $N\cong S^2$ with $\{z_\infty,z_{-\infty}\}$ as fixed-point set.
	\end{enumerate}
\end{thm}
\begin{proof}
Since $H$ is proper and non-constant, there exists an embedded circle $L$ of $H$ by Sard's theorem and hence a maximal orbit cylinder by Lemma~\ref{l:emb}. If there exists a resonant maximal cylinder $c\mapsto L_{c}$, $c\in(c_-,c_+)$, then applying Lemma~\ref{l:main} to the forward and backward maximal orbit cylinders through some $L_{c_0}$ we get item (ii) since $N$ is connected.	
\end{proof}

\begin{thm}\label{t:unbounded}
Let $(N,\omega)$ be a symplectic surface without boundary having finite symplectic area and such that no connected component of $N$ is diffeomorphic to the two-sphere. Let $H:N\to\R$ be a proper Hamiltonian. If $H$ is unbounded from above, then for every $c_0\in\R$ there exists a neighborhood $V\subset \{H> c_0\}$ of $\{H=+\infty\}$ with smooth, non-empty boundary such that $\p V$ is the union of non-resonant circles. A similar statement holds when $H$ is unbounded from below.
\end{thm}
\begin{proof}
Since $H$ is unbounded from above, there exists a regular value $c\geq c_0$ of $H$. Since $H$ is proper, the components $L_c^{(1)},\ldots,L_c^{(k)}$ of $\{H=c\}$ are embedded circles. Let
\[
S:=\{i\in \{1,\ldots,k\}\ |\ \text{the maximal forward cylinder of $L_c^{(i)}$ is resonant}\}.
\]
For all $i\in S$, there exists by Lemma~\ref{l:main} an embedded disc $D_i$ with boundary $L_c^{(i)}$ contained in $\{H\geq c\}$. We consider two cases according to whether $S\neq\{1,\ldots,k\}$ or not. If $S\neq\{1,\ldots,k\}$, then for all $i\notin S$, let $L_{c_i}^{(i)}$, $c_i\geq c$, be a non-resonant circle in the maximal forward cylinder of $L_c^{(i)}$, and let $U_{c,c_i}^{(i)}$ be the portion of the cylinder between $L_c^{(i)}$ and $L_{c_i}^{(i)}$. The statement of the theorem follows taking
\begin{equation*}
V:=\{H\geq c\}\setminus\Big[\Big(\bigcup_{i\in S}D_i\Big)\cup\Big(\bigcup_{i\notin S}U_{c,c_i}^{(i)}\Big)\Big].
\end{equation*}
If $S=\{1,\ldots,k\}$, then let us consider a regular value $c'> \max H|_{D_i}$ for all $i=1,\ldots,k$ which exists since $H$ is proper and unbounded from above. We denote by  $L_{c'}^{(1)},\ldots,L_{c'}^{(k')}$ the connected components of $\{H=c'\}$. Since $N$ does not contain connected components homeomorphic to the two-sphere, the maximal orbit cylinder through $L_{c'}^{(j)}$ contains a non-resonant circle $L_{c'_j}^{(j)}\subset\{H=c_j'\}$. By the choice of $c'$, we have $c'_j\geq c$ and we define $S':=\{j\in\{1,\ldots,k'\}\ |\ c'_j\geq c'\}$. In this case the statement of the theorem follows taking
\[
V:=\Big(\bigcup_{j\notin S'}U_{c'_j,c'}^{(j)}\Big)\cup\{H\geq c'\}\setminus\bigcup_{j\in S'}U_{c',c'_j}^{(j)}.\qedhere
\]\end{proof}

\begin{thm}\label{t:z-}
Let $(N,\omega)$ be a symplectic surface without boundary and let $H:N\to\R$ be a Hamiltonian having a strict local minimum at $z_-\in N$ which is degenerate. For every neighborhood $U$ of $z_-$ there exists a neighborhood $V\subset U$ of $z_-$ such that $\p V$ is a union of non-resonant circles. If $z_-$ is isolated in the set of  critical points of $H$, then $V$ can be taken to be an embedded disc.
\end{thm}
\begin{proof}
We can assume without loss of generality that $N$ is connected. Since $z_-$ is a strict local minimum, we can choose $c>H(z_-)$ such that $\{H\leq c\}\cap U$ is compact. By Sard's theorem, we can suppose that $c$ is a regular value of $H$ and let $L_c^{(1)},\ldots, L_c^{(k)}$ be the connected components of $\{H=c\}\cap U$. Let \[
S:=\{i\in \{1,\ldots,k\}\ |\ \text{the maximal backward cylinder of $L_c^{(i)}$ is resonant}\}.
\]
For each $i\in S$ there exists a closed embedded disc $D_i\subset \{H\leq c\}\cap U$ containing the maximal cylinder. Since $z_-$ is degenerate, $z_-\notin D_i$ by Lemma~\ref{l:main}. If $S=\{1,\ldots,k\}$, then $U':=(\{H\leq c\}\cap U)\setminus\bigcup_{i=1}^kD_i$ is both open in $U$ and compact. Since $N$ is connected and $z_-\in U'$, there holds $U'=N$ which is a contradiction since $D_i\subset N\setminus U'$ for all $i=1,\ldots,k$. Therefore $S\neq\{1,\ldots,k\}$ and we can construct the required $V$ as in the proof of Theorem~\ref{t:unbounded}. If $z_-$ is also isolated as critical point then a short topological argument shows that $\{H\leq c\}\cap U$ can be taken to be diffeomorphic to a disc.
\end{proof}

\begin{thm}\label{t:L'}
	Let $(N,\omega)$ be a symplectic surface without boundary, and let $H:N\to\R$ be a Hamiltonian having an isolated minimum at an embedded circle $L'\subset N$. Then for every neighborhood $U'$ of $L'$ there exists a neighborhood $V\subset U'$ of $L'$ such that $\p V$ is a union of non-resonant circles. If $L'$ is also isolated in the set of critical points of $H$, then $V$ can be taken to be a tubular neighborhood of $L'$.
\end{thm}
\begin{proof}
The argument is identical to that of Theorem~\ref{t:z-} using that $L'\cap D_i=\varnothing$ for all $i\in S$, where the discs $D_i$ are defined as in the proof above.
\end{proof}
\begin{thm}
\label{t:ultimo}
Let $(N,\omega)$ be a symplectic surface without boundary and let $H:N\to\R$ be a proper Hamiltonian having a non-degenerate saddle point at $z_1\in N$ with $c:=H(z_1)$. Consider a positive chart $(x,y):U\to (-\delta_0,+\delta_0)^2$ centered at $z_1$ such that $H(x,y)=c+xy$ for some $\delta_0>0$. For every $\epsilon<\delta<\delta_0$ there exist numbers $\epsilon_1,\epsilon_2,\epsilon_3,\epsilon_4\in(0,\epsilon)$ such that the following four sets belong to non-resonant circles
\begin{align*}
B_1&:=\big\{ H(x,y)=c+\delta\epsilon_1,\ x>0,\ y>0\big\},&B_2&:=\big\{ H(x,y)=c-\delta\epsilon_2,\ -x>0,\ y>0\big\},\\
B_3&:=\big\{ H(x,y)=c+\delta\epsilon_3,\ -x>0,\ -y>0\big\},
&B_4&:=\big\{ H(x,y)=c-\delta\epsilon_4,\ x>0,\ -y>0\big\}.
\end{align*}
\end{thm}

\begin{proof}
Define $C_{\epsilon'}:=\{H(x,y)=c+\delta\epsilon',\ x>0,\ y>0\}\subset N$ for some $\epsilon'\in(0,\epsilon)$ such that $c+\delta\epsilon'$ is a regular value of $H$ which exists by Sard's theorem. Let $L_{c+\delta\epsilon'}$ be the embedded circle containing $C_{\epsilon'}$. By Lemma~\ref{l:main} the maximal backward orbit cylinder through $L_{c+\delta\epsilon'}$ contains a non-resonant circle $L_{c+\delta\epsilon_1}$ for some $\epsilon_1\in(0,\epsilon']$. In a similar manner we construct $\epsilon_2,\epsilon_3,\epsilon_4$.
\end{proof}

\section{The Moser twist theorem: Trapping regions and Poincar\'e--Birkhoff orbits}

\label{s:Moser}

In this section we recall how to use non-resonant circles to prove the existence of trapping regions and periodic orbits for small perturbations of an autonomous Hamiltonian system given by a proper Hamiltonian $H:N\to\R$ on a symplectic surface $(N,\omega)$ without boundary. More precisely, for natural numbers $k<l$ consider a one-parameter family of time-dependent Hamiltonian functions
\begin{equation}\label{e:Hs}
H_s:N\times\T\to\R,\qquad H_s(z,t):=s^kH(z)+s^lR_s(z,t),\quad \forall\,(z,t)\in N\times \T, 
\end{equation}
where $R_s:N\times \T\to\R$ is a perturbation depending smoothly on the parameter $s\in[0,s_1)$ and $2\pi$-periodically on the time $t\in\T$. Let $X_{s,t}$ be the time-dependent Hamiltonian vector field associated with the function $z\mapsto H_s(z,t)$ and consider the map $\varphi_{t_0,t_1,s}:N_s\to N$ such that $\varphi_{t_0,t_1,s}(z)=z(t_1)$ where $z(t)$ is the unique solution of
\begin{equation}\label{e:nonaut}
\dot z(t)=X_{s,t}(z(t))
\end{equation}
with $z(t_0)=z$. The open set $N_s$ where $\varphi_{t_0,t_1,s}$ is defined contains an arbitrary compact subset of $N$ when $s$ is sufficiently small since $H_s$ depends periodically on time. 

An application of the Moser twist theorem and of the Poincar\'e--Birkhoff theorem yields the following classical result.
\begin{thm}\label{t:Moser}
Let $L\subset N$ be a non-resonant circle for $H$. For all neighborhoods $U$ of $L$, there exists $s_*\in(0,s_1)$ and a neighborhood $U'$ of $L$ contained in $U$ with the property that for all $s\in[0,s_*)$ there holds
\begin{enumerate}[(1)]
\item there exist infinitely many solutions of \eqref{e:nonaut} whose period is a multiple of $2\pi$ and whose free-homotopy class is a multiple of $L$,
\item every solution of \eqref{e:nonaut} starting in $U'$ at some time $t\in\R$ stays in $U$ for all times.
\end{enumerate}
\end{thm}
\begin{proof}
Let $(r,\theta)\in I\times \T$ be action-angle coordinates of $\omega$ in a tubular neighborhood of $L$ where $I$ is some open interval. These coordinates can be chosen so that the level sets of $r$ are embedded circles of $H$ and $L=\{r=r_0\}$ for some $r_0\in I$ \cite{Arnold:1978aq}. Let $\varphi_s:=\varphi_{0,2\pi,s}$. There exists a function $f:I\to\R$ so that in these coordinates we can write $\varphi_s(r,\theta)=(r_s,\theta_s)$, where
\begin{equation}\label{e:rtheta}
\begin{aligned}
r_s&=r+O(s^l),\\
\theta_s&=\theta+s^kf(r)+O(s^l).
\end{aligned}
\end{equation}
The condition of $L$ being non-resonant is equivalent to $f'(r_0)\neq0$ \cite{AB21}. Therefore there exists a sub-interval $[r_-,r_+]$ containing $r_0$ in its interior such that $f'(r)\neq 0$ for every $r\in[r_-,r_+]$ and $\{(r,\theta)\in [r_-,r_+]\times \T\}\subset U$.

Since $\varphi_s$ is a Hamiltonian diffeomorphism, it is exact symplectic. In other words, there exists a function $W_s:I\times\T\to\R$ such that
\[
\varphi_s^*(r_s\diff\theta_s)=r\diff\theta+\diff W_s.
\]
Thus we can apply the Moser twist theorem \cite{Mos62} at $r_-$ and $r_+$ to find two $\varphi_s$-invariant closed curves $\beta_{s,r_-}$, $\beta_{s,r_+}$ arbitrarily close to $\{r=r_-\}$, respectively $\{r=r_+\}$, for $s$ small enough. Thus we can assume that $\beta_{s,r_-}$ and $\beta_{s,r_+}$ bound an annular region $A_s\subset U$ which is invariant under $\varphi_s$ and contains $U':=\{(r,\theta)\in[\tilde r_-,\tilde r_+]\times \T\}$ in its interior for some $\tilde r_-<r_0<\tilde r_+$.

Let us show (2) in the statement of the theorem. For every $t\in\R$ let $A_{t,s}=\varphi_{0,t,s}(A_s)$ so that there holds
\[
\varphi_{t,t_1,s}(A_{t,s})=A_{t_1,s}\qquad \forall\,t,t_1\in\R.
\]
Since $H_s$ depends periodically on time and $r$ is constant along the flow of $H_0$, there holds $U'\subset A_{t,s}\subset U$ for all $s$ small enough and for all $t\in\R$. Thus a solution of \eqref{e:nonaut} passing through $U'$ stays in $U$ for all times.

Let us show (1) in the statement of the theorem. Consider the restriction of $\varphi_s$ to the closed annulus $A_s$. Recall the definition of the rotation number of a point $z\in A_s$
\[
\rho_s(z) := \lim_{n\to +\infty} \frac{\tilde \varphi^n_s(z)-z}{n}
\]
where $\tilde \varphi_s$ denotes a lift of $\varphi_s$ to the universal cover $\tilde A_s$ of $A_s$. If $z_\pm\in\beta_{s,r_\pm}$, then \eqref{e:rtheta} together with the fact that $\beta_{s,r_\pm}$ is close to $\{r=r_\pm\}$ implies that
\[
\rho_{s,\pm}:=\rho_s(z_\pm) = s^k2\pi f(r_\pm) + o(s^k)
\]
In particular, for $s$ small enough the rotation numbers $\rho_{s,-}$ and $\rho_{s,+}$ are different and by the Poincar\'e--Birkhoff fixed point theorem \cite[Theorem 7.1]{Gole:2001vx} the map $\varphi_s|_{A_s}$ has a periodic point with rotation number $p/q$ for every rational number $p/q$ lying between $\rho_{s,-}$ and $\rho_{s,+}$.
\end{proof}

Combining Theorem~\ref{t:Moser} with Theorem~\ref{t:dichotomy} we get the existence of infinitely many periodic orbits and trapping regions for $H_s$ if $(N,\omega)$ is a connected, symplectic surface of finite area such that either $N\not\cong S^2$ or $N\cong S^2$ and $H$ does not give a Hamiltonian circle action. Combining Theorem~\ref{t:Moser} with Theorems~\ref{t:unbounded},~\ref{t:z-}, and~\ref{t:L'}, we get more precise information on the location of the periodic orbits and the trapping regions.
\begin{thm}\label{t:three}
If $H$ is unbounded from above (resp.~below), or has a degenerate strict local minimum (resp.~maximum) $z_0\in N$, or has a strict local minimum (resp.~maximum) at an embedded circle $L'$, then arbitrarily small neighborhoods of $\{H=+\infty\}$, $z_0$ or $L'$ are trapping regions for the orbits of \eqref{e:nonaut} and contain infinitely many periodic orbits of \eqref{e:nonaut} for every $s$ small enough. \qed
\end{thm}
Finally, we apply Theorem~\ref{t:ultimo} to study the stability properties of non-degenerate saddle points of $H$. 

\begin{thm}\label{t:saddle}
Let $(N,\omega)$ be a symplectic surface without boundary, $H:N\to\R$ a proper Hamiltonian with a non-degenerate saddle point $z_1$. Under the notation of Theorem~\ref{t:ultimo}, for every $\epsilon<\delta<\delta_0$ define 
\[
U_\delta:=\big\{(x,y)\in(-\delta,\delta)^2\big\},\quad A_{\delta,\epsilon}^+:=\big\{|x|=\delta,\ |y|\leq \epsilon\big\},\quad A_{\delta,\epsilon}^-:=\big\{|y|=\delta,\ |x|\leq \epsilon\big\}.
\]
Then for every $\epsilon<\delta<\delta_0$ there exists $\delta'<\epsilon$ and $s_*>0$ such that for all $s<s_*$ every orbit $z$ of \eqref{e:nonaut} with $z(t_0)\in U_{\delta'}$ for some $t_0\in \R$ has the following property: There exist $t_-\in[-\infty,t_0)$ and $t_+\in(t_0,+\infty]$ such that $z(t)\in U_\delta$ for all $t\in(t_-,t_+)$. Moreover, either $t_\pm=\pm\infty$ or $z(t_\pm)\in A^\pm_{\delta,\epsilon}$.  
\end{thm}
\begin{proof}
Given $\epsilon$ and $\delta$, Theorem~\ref{t:ultimo} yields the sets $B_1$, $B_2$, $B_3$, $B_4$ which are parts of non-resonant circles. Applying Theorem~\ref{t:Moser} to these non-resonant circles we obtain a value $s_*$ and $\delta'<\epsilon$, and tubular neighborhoods of these circles which cannot be intersected by solutions $z$ of \eqref{e:nonaut} with parameter $s<s_*$ such that $z(t_0)\in U_{\delta'}$ for some $t_0$. Let $V$ be the union of these tubular neighborhoods and let $V'$ the connected components of $U_\delta\setminus V$ containing $U_{\delta'}$. Then
\[
\p V'\cap \p U_\delta\subset A^-_{\delta,\epsilon}\cup A^+_{\delta,\epsilon}. 
\]
Thus either $z$ stays in $U_\delta$ for all times larger than $t_0$ or it must have a first intersection time larger than $t_0$ with $\p U_\delta$ at $A^-_{\delta,\epsilon}\cup A^+_{\delta,\epsilon}$. Notice however that $X_{t,s}$ points inside $U_\delta$ on $A^-_{\delta,\epsilon}$ and outside of it on $A^+_{\delta,\epsilon}$ if $s$ is small enough. Therefore $z$ can only exit $U_\delta$ at $A^+_{\delta,\epsilon}$ and enter at $A^-_{\delta,\epsilon}$.  
\end{proof}
\section{Applications to magnetic systems via the Hamiltonian normal form}\label{s:app}
In this section we see how the normal form established in \cite{Castilho:2001,AB21} allows us to apply the results of the previous section to magnetic systems at low speeds and prove Theorems~\ref{t:main} and~\ref{t:main2}.

Let $(g,b)$ be a magnetic system on a closed, connected, oriented surface $M$. Equation~\eqref{e:ODE} induces a flow on $TM$, usually referred to as the \textit{magnetic flow}, which preserves the set of vectors with fixed speed $S_s M := \{(q,v)\in TM\ |\ |v|_q=s\}$, where $|\cdot|$ is the norm associated with $g$. We recall now that the flow on $TM$ is Hamiltonian. To this purpose let $\lambda\in \Omega^1(TM)$ be the Hilbert form of $g$, namely the pull-back of the standard Liouville one-form on $T^*M$ by means of the isomorphism $TM\to T^*M$ induced by the metric. We define the \textit{twisted symplectic form}
\[
\omega_{(g,b)} := \diff \lambda - \pi^*(b\mu),
\]
where $\pi:TM\to M$ is the foot-point projection and $\mu$ is the area form induced by $g$ and the given orientation, and consider the kinetic Hamiltonian 
\[
H_\kin :TM \to \R,\qquad H_{\kin}(q,v) = \frac{1}{2}|v|^2_q.
\]
It is a classical fact that the Hamiltonian flow $\Phi_{(g,b)}:\R\times TM\to TM$ induced by the pair $(H_\kin,\omega_{(g,b)})$ is the magnetic flow of the pair $(g,b)$. This implies that the restriction of the magnetic flow to each $S_s M$ is everywhere tangent to the characteristic line distribution 
\[
\ker (\omega_{(g,b)}|_{S_s M})\subset T (S_s M).
\]
Setting $SM:=S_1M$ and defining the rescaling map $F_s:SM\to S_sM$, $F_s(q,v)=(q,sv)$, we get
\[
\omega_s:=F_s^*\omega_{(g,b)}|_{S_sM}=s\diff\lambda-\pi^*(b\mu),
\]
where on the right-hand side we have also denoted by $\lambda$ and $\pi^*(b\mu)$ the restrictions of these objects to $SM$. Therefore the map $F_s$ identifies integral curves of the characteristic distribution $\ker\omega_s$ with magnetic geodesics, namely solutions of \eqref{e:ODE}, up to time reparametrization.

Let $U\subset M$ be an open set on which $b$ does not vanish and such that $\pi:SU\to U$ admits
a section $W:U\to SU$. Thus $\omega:=b\mu$ is a symplectic form and $(U,b\mu)$ is a symplectic surface without boundary with finite area. Moreover the angular function $t:SU\to \T$ associated with $W$ yields a trivialization $\tau:U\times \T\to SU$. Following \cite[Theorem 1.1]{AB21} we see that for every open set $U'$ such that $\overline{U'}\subset U$, there exist $s_0>0$ and an isotopy of embeddings $\Psi_s:U'\times\T\to SU$, $s\in[0,s_0)$ with $\Psi_0=\tau|_{U'\times\T}$ such that
\[
\Psi_s^*\omega_s=\diff(H_s\diff t)-\pi^*(b\mu),
\]
where $H_s:U'\times\T\to\R$ is a path of functions such that
\begin{equation}\label{e:H1}
H_s(z,t)=-\frac{s^2}{2b(z)}+s^3R_s(z,t),\qquad \forall\,(z,t)\in U'\times\T
\end{equation}
for some remainder $R_s:U'\times\T$. If $b$ is constant, then we can choose $\Psi_s$ so that
\begin{equation}\label{e:H2}
H_s(z,t)=-\frac{s^2}{2b}-\frac{s^4}{(2b)^3}K(z)+s^5\tilde R_s(z,t),\qquad\forall\,(z,t)\in U'\times\T,
\end{equation}
where $K:M\to\R$ is the Gaussian curvature of $g$ and $\tilde R_s:U'\times\T\to\R$ is some remainder. A short computation shows that the curves $\Gamma:(t_0,t_1)\to U'\times\T$, $\Gamma(t)=(z(t),t)$ where $z$ is a solution of the non-autonomous Hamiltonian system with Hamiltonian $H_s$ and symplectic form $\omega=b\mu$ are exactly the integral curves of the characteristic foliation of $\ker(\diff(H_s\diff t)-\pi^*(b\mu))=\ker(\Psi_s^*\omega_s)$. 
\begin{rmk}\label{r:bb}
Up to a constant factor, the leading term in \eqref{e:H1} is given by the Hamiltonian $b^{-1}:U'\to\R$. The Hamiltonian flow of $b^{-1}$ with symplectic form $b\mu$ is the same as the Hamiltonian flow of $\tfrac12b^{-2}$ with symplectic form $\mu$. Thus non-resonant circles for $b^{-1}$ with symplectic form $b\mu$ are the same as non-resonant circles for $b^{-2}$ with symplectic form $\mu$. \qed
\end{rmk}
\begin{rmk}\label{r:psis}
If $t\mapsto z(t)$ is a periodic orbit for $H_s$, then $\gamma(t):=\pi(\Psi_s(z(t),t))$ is a periodic magnetic geodesic with speed $s$ freely homotopic to $z$ since $\pi\circ\Psi_0=\pi$. Moreover, if $U'$ is an open set with $\overline{U'}\subset U$ having the property that every solutions of the non-autonomous Hamiltonian system of $H_s$ passing through $U'$ stays in $U$ for all times, then up to slightly modifying $U$ and $U'$ the same holds for magnetic geodesics with small speed $s$ since $\Psi_s$ is close to the trivialization $\Psi_0=\tau$. \qed
\end{rmk}
We are in position to apply the results of Section~\ref{s:Moser} to the trajectories of $H_s$ and then transfer them to the solutions of \eqref{e:ODE} via the maps $F\circ\Psi_s$. 

\begin{proof}[Proof of Theorem~\ref{t:main}]
Let $(g,b)$ be a magnetic system which is non-resonant at zero speed. Let us assume first that $b$ is not constant and define the open set $N:=\{b\neq0\}$. The function $b^{-1}:N\to\R$ is proper and the symplectic area of $N$ with respect to $\omega=b\mu$ is finite. Then Theorem~\ref{t:dichotomy} implies that $b^{-1}$ has a non-resonant circle $L$: This is clear if $\{b=0\}$ is non-empty (as $N\neq S^2$ in this case) and it is true by Definition~\ref{d:non} together with Remark~\ref{r:bb} otherwise. By Theorem~\ref{t:Moser} applied to a neighborhood $U$ of the circle $L$, for $s$ small enough the function $H_s$ given in \eqref{e:H1} admits infinitely many periodic orbits freely homotopic to some multiple of $L$ and there is a neighborhood $U'$ of $L$ such that every orbit of $H_s$ starting in $U'$ stay in $U$ for all times. By Remark~\ref{r:psis} the same applies to magnetic geodesics for a small enough speed $s$. 

 If $b$ is a non-zero constant, then by Definition~\ref{d:non} and Theorem~\ref{t:dichotomy} the Gaussian curvature $K:M\to\R$ admits a non-resonant circle $L$ with respect to the symplectic form $\mu$. In this case the leading term of $H_s$ given in \eqref{e:H2} is equal to $K$ up to constants and we can run the same argument as in the case of a non-constant $b$ using Theorem~\ref{t:Moser} to finish the proof of Theorem~\ref{t:main}.
\end{proof}
\begin{proof}[Proof of Theorem~\ref{t:main2}]
Using Remark~\ref{r:bb} and~\ref{r:psis} the result follows from Theorem~\ref{t:three} applied to $H_s$ given in \eqref{e:H1}, \eqref{e:H2}. Indeed observe that if the set $\{b=0\}$ is non-empty, then the Hamiltonian 
$b^{-1}$ is proper and unbounded on the non-empty open set $\{b\neq0\}$.
\end{proof}
Using the notation of Theorem~\ref{t:saddle}, we get a result for saddle points of $b$ or $K$ as promised in Remark~\ref{r:saddle}.
\begin{thm}\label{t:saddlemag}
Let $z_1\in\{b\neq0\}$ be a non-degenerate saddle point of the function $-b^{-1}$. For every $\epsilon<\delta<\delta_0$ there exists $\delta'$ and $s_*$ such that for all $s<s_*$ every magnetic geodesic $\gamma$ with speed $s$ and $\gamma(t_0)\in U_{\delta'}$ for some $t_0\in \R$ has the following property: There exist $t_-\in[-\infty,t_0)$ and $t_+\in(t_0,+\infty]$ such that $\gamma(t)\in U_\delta$ for all $t\in(t_-,t_+)$. Moreover, either $t_\pm=\pm\infty$ or $\gamma(t_\pm)\in A^\pm_{\delta,\epsilon}$. A similar statement holds if $b$ is a non-zero constant and $z_1$ is a non-degenerate saddle point for $-K$.
\end{thm}
\begin{proof}
Using Remark~\ref{r:bb} and~\ref{r:psis} the result follows from Theorem~\ref{t:saddle} applied to $H_s$ given in \eqref{e:H1}, \eqref{e:H2}.
\end{proof}

%%%%%%%%%%%%%
%%%%%%%%%%%%%
%%%%%%%%%%%%%

\section{The resonant case}\label{s:resonant}

In this last section, we collect some observations about magnetic systems $(g,b)$ on $S^2$ that are resonant at zero speed for which the main results in Section~\ref{ss:main} cannot be applied. Let us have a look first at the case in which $(g,b)$ is rotationally symmetric. This means that there are spherical coordinates $(r,\phi)\in(r_-,r_+)\times\T$ on the complement of the north and south pole such that $g=\diff r^2+a(r)^2\diff \phi^2$ and $b=b(r)$ for some functions $a,b:(r_-,r_+)\to\R$. A closer inspection of the proof of \eqref{e:H1} and \eqref{e:H2} given in \cite[Section 3]{AB21} shows that $H_s$ is independent of $\phi$ in this case. Thus we have the formula
\begin{equation}\label{e:xst}
X_{s,t}=\frac{\partial_r H_s}{ab}\partial_\phi
\end{equation}
for the time-dependent Hamiltonian vector field of $H_s$. Therefore the level sets of the function $r$ with $b(r)\neq0$ are invariant under the flow of $H_s$ and using the isotopy $\Psi_s$ we find that magnetic geodesics with speed $s$ are trapped around each of these circles. By \eqref{e:xst} and \eqref{e:H1}, the dynamics of $H_s$ on such circles over an interval of time of length $2\pi$ is given by a rotation of angle 
\[
\rho_s(r):=-s^2\pi \frac{b'(r)}{a(r)b^3(r)}+O(s^3).
\]
Therefore, if $b'(r)\neq0$, the function $s\mapsto \rho_s(r)$ is non-constant for $s$ small. Thus there is a decreasing sequence $s_n$ converging to $0$ such that $\rho_{s_n}(r)$ is rational. By continuity $\rho_{s_n}(r_m)$ is rational for infinitely many values $(r_m)_{m\in\N}$ close to $r$. In particular, there are infinitely many magnetic geodesics with speed $s_n$ for all $n\in\N$. 

In the complementary case in which $b$ is a non-zero constant, we can show that the curvature $K$ cannot yield a non-constant circle action. 
\begin{thm}
\label{thm:curvaturenonresonant}
Let $g=\diff r^2+a(r)^2\diff\phi^2$ be a metric of revolution on $S^2$ where $a:[0,R]\to[0,\infty)$ with non-constant Gaussian curvature $K$. Then $K$ does not induce a Hamiltonian circle action on $(S^2,\mu)$.
\end{thm}
\begin{proof}
Up to rescaling $g$ we can assume that $\mathrm{area}_g(S^2)=4\pi$. Let $A:[0,R]\to\R$ be the primitive of $a$ with the property that $A(0)=-1$ and $A(R)=1$. Since $r=0,R$ correspond to the south and north pole of $S^2$, we also have $A'(0)=0=A'(R)$ and $A''(0)=1$, $A''(R)=-1$.

We have $\mu=a\diff r\wedge\diff\phi=\diff A\wedge\diff\phi$. Therefore, $K$ induces a non-constant Hamiltonian circle action if and only if
\[
K=c_1A+c_2
\]
for some constants $c_1,c_2\in\R$ with $c_1\neq0$. Since $K=-(A')^{-1}A'''$, we can integrate this equation to
\begin{equation}\label{e:c1}
2c_1A''+(c_1A+c_2)^2+c_3=0
\end{equation}
for some constant $c_3$. Using the boundary conditions on $A$ and $A''$ we get
\[
2c_1+(c_2-c_1)^2+c_3=0,\quad -2c_1+(c_2+c_1)^2+c_3=0.
\]
Eliminating $c_3$, we find $4c_1=(c_2+c_1)^2-(c_2-c_1)^2$. Since $c_1\neq0$, this is equivalent to $c_2=1$. Consequently $c_3=-1-c_1^2$. Substituting $c_2$ and $c_3$ in \eqref{e:c1} and dividing by $c_1$, we arrive at
\[
2A''+c_1A^2+2A-c_1=0.
\]
Multiplying this equation by $A'$ and integrating once again, we find
\[
(A')^2+\tfrac13c_1A^3+A^2-c_1A=c_4.
\]
Using the boundary conditions for $A$ and $A'$, we get
\[
\tfrac13c_1+1-c_1=-\tfrac13c_1+1+c_1,
\]
which is equivalent to $c_1=0$, a contradiction.
\end{proof}
\begin{cor}\label{c:rotinf}
Let $(g,b)$ be a magnetic system on $S^2$ such that $g$ is rotationally symmetric and $b$ is a non-zero constant. Then for every $s$ small enough there are infinitely many periodic magnetic geodesics of speed $s$.\qed
\end{cor}

Prompted by these results for rotationally symmetric magnetic systems we can formulate the following conjecture for arbitrary resonant magnetic systems.

\begin{conj*}
Let $(g,b)$ be a magnetic system on $S^2$ which is resonant at zero speed. Then there exists a sequence $s_n\to 0$ such that there are infinitely many periodic magnetic geodesics with speed $s_n$ for all $n\in\N$. 
\end{conj*}

Finally, we can ask what geometric conditions general magnetic systems $(g,b)$ on $S^2$ must satisfy in order to be non-resonant at zero speed. Surely, given $g$ one readily finds a positive, non-constant function $b$ such that $b^{-2}$ yields a Hamiltonian circle action on $(S^2,\mu)$. Thus we consider the other case in which $b$ is constant and $g$ is a metric whose Gaussian curvature $K$ yields a non-constant Hamiltonian circle action on $(S^2,\mu)$. Then $K$ is a Morse function with exactly one minimum and one maximum and by Lemma~\ref{l:ap}
\begin{equation}\label{e:muk}
\mu=\diff K\wedge \diff\phi\qquad \phi\in\R/T\Z,
\end{equation}
where
\[
T:=\frac{\mathrm{area}_g(S^2)}{\max K-\min K}.
\]
Integrating \eqref{e:muk} on $\{K\leq k\}$ for an arbitrary $k$ we get 
\begin{equation}\label{e:Kres0}
k=\min K+\frac{\max K-\min K}{\mathrm{area}_g(S^2)}\mathrm{area}_g\big(\{K\leq k\}\big),\qquad \forall\,k\in[\min K,\max K].
\end{equation}
Moreover, applying the Gauß--Bonnet theorem we see that 
\[
4\pi=\int_{S^2}K\mu=\int_{\min K}^{\max K}\int_0^TK\diff K\,\diff\phi=\frac12\big((\max K)^2-(\min K)^2)T=\frac{\max K+\min K}{2}\mathrm{area}(S^2,g),
\]
which yields the following necessary condition for $K$ to induce a non-constant circle action
\begin{equation}\label{e:Kres}
\frac{\max K+\min K}{2}=	\frac{4\pi}{\mathrm{area}_g(S^2)}.
\end{equation}
\bibliography{_biblio}
\bibliographystyle{plain}

\end{document}